\documentclass[12pt,reqno]{amsart}

\usepackage[arrow,matrix,curve]{xy}

\usepackage[dvips]{graphicx} 
\usepackage{mathtools}

\usepackage{amssymb, latexsym, amsmath, amscd, array, hyperref, esint
}

\usepackage{breakurl}   %this allows for url breaks in pdf

\makeatletter
\@namedef{subjclassname@2020}{\textup{2020} Mathematics Subject Classification}
\makeatother

\newtheorem{theorem}{Theorem}[section]
\newtheorem{lemma}[theorem]{Lemma}

\theoremstyle{definition}

\newtheorem{remark}[theorem]{Remark}

\numberwithin{equation}{section}
\newcommand\eg{{\it e.g.}}

\DeclareMathOperator{\area}{{\rm area}} 
 
\DeclareMathOperator{\length}{{\rm length}} 
\DeclareMathOperator{\sys}{{\rm sys}}

\DeclareMathOperator{\arcsinh}{arcsinh}

\long\def\forget#1\forgotten{} %

\numberwithin{equation}{section}

\title[Logarithmic systolic growth in every genus]{Logarithmic systolic growth for hyperbolic surfaces in every genus}

\author[M. Katz]{Mikhail G. Katz} \address{Department of
  Mathematics, Bar Ilan University, Ramat Gan 5290002 Israel}
\email{katzmik@math.biu.ac.il}

\author[S. Sabourau]{St\'ephane Sabourau}
\address{\parbox{\linewidth}{Univ Paris Est Creteil, CNRS, LAMA, F-94010 Creteil, France \\
Univ Gustave Eiffel, LAMA, F-77447 Marne-la-Vall\'ee, France}}
\email{stephane.sabourau@u-pec.fr}


\begin{document}

\begin{abstract}
%By the work of Buser and Sarnak, Gromov, Schmutz Schaller, and others, it is known that the $\limsup$ as $g\to\infty$ of the ratio $\text{systole}/\log(g)$ for maximal hyperbolic surfaces of genus~$g$ is at least~$4/3$.  We show that the $\liminf$ is at least~$19/120$.
More than thirty years ago, Brooks and Buser--Sarnak constructed sequences of closed hyperbolic surfaces with logarithmic systolic growth in the genus.
Recently, Liu and Petri showed that such logarithmic systolic lower bound holds for every genus (not merely for genera in some infinite sequence) using random surfaces.
In this article, we show a similar result through a more direct approach relying on the original Brooks/Buser--Sarnak surfaces.
\end{abstract}

%\doublespacing

\thispagestyle{empty}

%\huge


\maketitle
%\tableofcontents

%\today

\section{Introduction}

%Buser and Sarnak \cite{BS} give an argument to obtain asymptotic lower bounds for the systole in terms of the genus, for arbitrary genus.
%They do not provide all the details and we have not been able to fill them in.  
%Furthermore, their approach does not lead to an explicit lower bound for $\liminf_g$.

The systole of a closed hyperbolic surface~$S$, denoted by~$\sys(S)$,
is the length of a shortest (noncontractible) closed geodesic of~$S$.
By Mumford's compactness theorem, the systole function achieves a
maximum among all closed hyperbolic surfaces of a given genus.
By considering the area of a disk of radius~$\frac{1}{2} \sys(S)$, one
easily sees (see \cite[Lemma~5.2.1]{buser}) that the systole of every
closed hyperbolic surface~$S$ of genus~$g$ satisfies the upper bound
\begin{equation} \label{eq:2}
\sys(S) \leq 2 \log(4g-2).
\end{equation}

As far as lower bounds are concerned, Brooks~\cite{brooks88} and Buser--Sarnak~\cite{BS} constructed a sequence of closed hyperbolic surfaces~$S_{g_p}$ of genus~$g_p$ with
\begin{equation} \label{eq:gp}
g_p= (p^3-p) \nu +1
\end{equation}
obtained as congruence coverings %of degree~$\frac{1}{2} (p^3-p)$ 
of a fixed arithmetic surface~$S$ defined from a division quaternion algebra~$A$ %(see~\cite[Theorem~5.5.4]{kat} for a computation of the degree) 
such that
\begin{equation} \label{eq:log4}
\sys(S_{g_p}) \geq 4 \log(p)
\end{equation}
where $p$ is any odd prime and $\nu$ is  fixed and depends only on~$A$; see~\cite[Lemma~3.1]{brooks88} and~\cite[Eq.~(4.3)]{BS}.
%(Note that $\nu = \frac{1}{2} (g_S-1) \geq \frac{1}{2}$, where $g_S$ is the genus of the arithmetic surface~$S$).
This immediately leads to the following logarithmic systolic growth 
\begin{equation} \label{eq:4/3}
\sys(S_{g_p}) \geq \frac{4}{3} \log(g_p) - C
\end{equation}
where $C$ is a positive constant depending only on~$A$, but not on~$p$.

Combining the upper and lower bounds~\eqref{eq:2} and~\eqref{eq:4/3}, we obtain
\begin{equation} \label{eq:limsup}
\frac{4}{3} \leq \limsup_{g \to \infty} \max_{S \in \mathcal{M}_g} \frac{\sys(S)}{\log g} \leq 2
\end{equation}
where $\mathcal{M}_g$ is the moduli space of closed hyperbolic
surfaces of genus~$g$.  
See also Brooks~\cite{brooks}, Gromov~\cite{Gr96} and Schmutz Schaller~\cite{Sc93}.

The Brooks/Buser--Sarnak construction was generalized in~\cite{KSV} to
principal congruence covers of any closed arithmetic surface.
Makisumi \cite{Mak} proved that the multiplicative
constant~$\frac{4}{3}$ in~\eqref{eq:4/3} is optimal for these
congruence covers.  Other sequences of closed hyperbolic surfaces with
arbitrarily large genus satisfying a logarithmic growth have been
obtained in~\cite{Pet} and~\cite{PW} from graphs of large
(logarithmic) girth/systole. 
%All these systolic lower bounds are essentially obtained through arithmetic or probabilistic means and result in finite families of surfaces of fairly sparse genera. 
Still, these constructions provide families of surfaces of fairly sparse genera. 

Recently, Liu and Petri~\cite{LP} showed using a probabilistic method that for every~$c<\frac{2}{9}$, there exists a constant~$C>0$ such that for \emph{every} genus~$g \geq 2$ (not merely for genera in some infinite sequence), there exists a closed hyperbolic surface~$S_g$ of genus~$g$ with
\[
\sys(S_g) \geq c \log(g) - C.
\]
It follows that
\[
\liminf_{g \to \infty} \max_{S \in \mathcal{M}_g} \frac{\sys(S)}{\log g} \geq \frac{2}{9}.
\]
Note that their argument relies on an appropriate model of random surfaces and does not provide an explicit construction.

In the present article, we show a similar result, see
Theorem~\ref{theo:main}, without using probabilistic arguments (albeit
with a worse constant).  Instead, we use a more direct approach
relying on the original surfaces of Brooks/Buser--Sarnak.

\begin{theorem} \label{theo:intro}
For every $c<\frac{19}{120}$, there exists a constant~$C>0$ such that for every genus~$g \geq 2$, there exists a closed hyperbolic surface~$S_g$ of genus~$g$ with
\[
\sys(S_g) \geq c \log(g) - C.
\]
\end{theorem}

\forget
As mentioned above, the lower bounds developed through arithmetic
means result in families of surfaces of fairly sparse genera.  In the
present text, we consider the $\liminf_g$ of the maximum of the
ratio~$\sys/\log(g)$ over~$\mathcal{M}_g$, as opposed to its
$\limsup_g$ as in~\eqref{eq:limsup}.  More specifically, the question
we ask can be formulated as follows.  Do there exist constants $c,C>0$
such that for \emph{every} genus~$g \geq 2$ (not merely for genera in
some infinite sequence), there exists a closed hyperbolic
surface~$S_g$ of genus~$g$ with
\[
\sys(S_g) \geq c \log(g) - C
\]
with an explicit constant $c$ if possible?  We answer this question
affirmatively in Theorem~\ref{theo:main}, filling an apparent gap in
the literature on this subject, by showing the following result.

\begin{theorem} \label{theo:intro}
The following asymptotic bounds hold:
  \begin{equation*} %\label{eq:liminf}
\frac{19}{120} \leq \liminf_{g \to \infty} \max_{S \in \mathcal{M}_g} \frac{\sys(S)}{\log g} \leq 2.
\end{equation*}
\end{theorem}
\forgotten

Actually, Buser and Sarnak~\cite{BS} considered a similar question for
the systole of the Jacobian of a Riemann surface of large genus after
showing that there exists a sequence of Riemann surfaces of
arbitrarily large genus whose square of the Jacobian systole has a
logarithmic growth in the genus.
To show this holds in every genus, they mention that one could seek to
vary the parameters in the definition of the quaternion algebra~$A$
and to vary the congruence groups in the arithmetic construction, but
as they point out this seems rather complicated; see~\cite[p.~47]{BS}.
Instead, they rely on the circle method from additive analytic number
theory and Fay's gluing of Rieman surfaces~\cite[\S III]{Fay} to
construct Rieman surfaces whose Jacobian decomposes as the product of
Jacobians of lower dimension up to a small error term
(see~\cite[p.\;47]{BS}), which is enough for their purposes.
As mentioned in~\cite[p.\;127]{Pa14}, it seems likely (though this would
require some argument) that the homological systole of the underlying
hyperbolic surfaces thus-obtained have a logarithmic growth.
It is however unclear whether this construction yields an explicit
homological systolic lower bound (as we have not been able to fill in
all the details) and no such explicit bound for the Jacobian systole
is given in~\cite{BS} (only for the surfaces constructed as congruence
covers of an arithmetic surface).
Note also that, though the homological systole of these surfaces is
likely to have a logarithmic growth, their systole is equally likely
to go to zero.
However, since for hyperbolic surfaces of a fixed genus, the maximum
of the systole is equal to the maximum of the homological systole by
Parlier~\cite{par}, this would not be an obstacle to deriving a
logarithmic systolic growth in every genus.

To prove Theorem~\ref{theo:intro}, we rely on a more flexible construction using $\textrm{CAT}(-1)$ metrics rather than the more rigid hyperbolic metrics.
This change of setting is not restrictive for our purpose since the maximum of the systole on $\textrm{CAT}(-1)$ surfaces of fixed genus is achieved by a hyperbolic metric according to~\cite{JS}.
Along the way, we apply the prime gap theorem to control the growth of the genus in Brooks/Buser--Sarnak's congruence surfaces.
As a result, we obtain an explicit asymptotic systolic lower bound. \\

Recent advances in systolic geometry include \cite{Go23} and
\cite{Ka24}.

\medskip
\emph{Acknowledgement.} The authors would like to thank Hugo Parlier
for bringing the article~\cite{LP} to their attention after receiving
a first version of the present article.

\section{Preliminaries}

For future reference, let us recall some classical formulas in
hyperbolic geometry.  The area and the perimeter of a disk~$D(r)$ of
radius~$r$ in the hyperbolic plane are given by
\begin{align}
\area D(r) & = 4 \pi \sinh^2 \left( \frac{r}{2} \right) \label{eq:area} \\
\length \partial D(r) & = 2 \pi \sinh(r) \label{eq:length}
\end{align}
and the area of a closed hyperbolic surface~$S$ of genus~$g$ is given by
\begin{equation} \label{eq:A}
\area(S) = 4\pi (g-1).
\end{equation}

Due to the arithmetic nature of the construction of the Brooks/Buser--Sarnak surfaces, we will also need the following prime gap theorem.
%Given a prime number~$p$, denote by~$p'$ the next prime number.

\begin{theorem}[Prime gap theorem~\cite{BHP}] \label{theo:primegap}
There exist~$\lambda \geq 1$  and~\mbox{$\theta \in (0,1)$} such that the gap between any prime~$p$ and the next prime~$p'$ satisfies
%for every prime number~$p$, we have
\begin{equation} \label{eq:theta}
p' -p \leq \lambda p^\theta.
\end{equation}
For this bound, we can take~$\theta = \frac{21}{40} = 0.525$. 
\end{theorem}

\begin{remark}
For $p$ large enough, the multiplicative constant~$\lambda \geq 1$ can be taken arbitrarily close to~$1$; see~\cite{BHP}.
\end{remark}

\section{Logarithmic systolic growth}

We will use the notation~$A_g \simeq B_g$ for $A_g = B_g (1+o(1))$ as $g$ goes to infinity (similarly with $A_g \lesssim B_g$ and $A_g \gtrsim B_g$).

\begin{theorem} \label{theo:main}
For every $g \geq 2$, there exists a closed hyperbolic surface~$S_g$ of genus~$g$ with 
\begin{equation} \label{eq:loggrowth}
\sys(S_g) \gtrsim \frac{1-\theta}{3} \log(g)
\end{equation}
where $\theta \in (0,1)$ is the exponent in~\eqref{eq:theta}.
\end{theorem}

\begin{remark}
With the value of~$\theta$ given in Theorem~\ref{theo:primegap}, the coefficient~$\frac{1-\theta}{3}$ can be taken equal to~$\frac{19}{120} = 0.18333... \geq \frac{1}{7}$.
\end{remark}

\begin{proof}
For every odd prime~$p$, consider a closed hyperbolic surface~$S_{g_p}$ of genus~$g_p$ satisfying~\eqref{eq:log4}.
Let $r$ and~$d$ be two positive reals (to be determined later) such that $r+d<\frac{1}{4} \sys(S_{g_p})$. 
Consider a maximal collection~$(D_i)_{1 \leq i \leq N}$ of disjoint disks~$D_i$ of radius~$(r+d)$ in~$S_{g_p}$.
Denote by~$x_i$ the center of~$D_i$.
Since the collection of disks~$D_i$ is maximal, the disks~$2D_i$ of radius~$2(r+d)$ centered at~$x_i$ cover~$S_{g_p}$.
Since their radius is less than~$\frac{1}{2} \sys(S_{g_p})$, the disks~$2D_i$ can be isometrically embedded in the hyperbolic plane onto a disk~$2D$ of radius~$2(r+d)$.
It follows that the number~$N$ of points~$x_i$ satisfies
\[
N \cdot \area(2D) \geq \area(S_{g_p}).
\]
By the relations~\eqref{eq:area} and~\eqref{eq:A}, we immediately obtain
\begin{equation} \label{eq:N}
N \geq \frac{g_p-1}{\sinh^2(r+d)}.
\end{equation}

Now, remove $2k$ disks~$D(x_i,r)$ from~$S_{g_p}$ with $k \leq \frac{N}{2}$, and pairwise glue together the boundary components of the resulting surface.
This gives rise to a closed $\textrm{CAT}(-1)$ surface~$\bar{S}_{g_p + k}$ of genus $g_p + k$. \\

Before showing that the surface~$\bar{S}_{g_p+k}$ satisfies the logarithmic systolic growth of~\eqref{eq:loggrowth}, we will first show that every genus large enough can be attained as the genus of a surface~$\bar{S}_{g_p+k}$.
For that, for $p$ large enough, we want $g_p + \lfloor \frac{N}{2} \rfloor$ to be at least $g_{p'}$ to recover every genus between~$g_p$ and~$g_{p'}$, when $k$ runs between~$0$ and~$\frac{N}{2}$ (recall that $p'$ is the prime right after~$p$).
By~\eqref{eq:N}, it is enough to have
\begin{equation} \label{eq:sinh}
\frac{g_p -1}{2 \sinh^2(r+d)} \geq g_{p'} - g_p +1.
\end{equation}
Since $g_p =(p^3-p) \nu +1$, the genus gap~$g_{p'} - g_p$ satisfies
\begin{align}
g_{p'} - g_p + 1& = (p'^3-p^3) \nu + (p'-p) \nu +1 \nonumber \\
 & \leq 3 \nu \lambda p^{2+\theta} + 3 \nu \lambda^2 p^{1+2\theta} + \nu \lambda^3 p^{3\theta} + \nu \lambda p^\theta + 1 \nonumber \\
 & \leq 8 \nu \lambda^3 p^{2+\theta} + 1 \label{eq:8}
\end{align}
after making use of the bound $p' \leq p+ \lambda p^\theta$ (see Theorem~\ref{theo:primegap}), expanding the power, and bounding each term $p^{1+2\theta}$, $p^{3\theta}$ and $p^\theta$ by~$p^{2+\theta}$.
Thus, by~\eqref{eq:sinh}, \eqref{eq:gp} and~\eqref{eq:8}, it is enough to have
\[
\sinh^2(r+d) \leq p^{1-\theta} \, \frac{(1-p^{-2}) \nu}{2(8 \nu \lambda^3+p^{-2-\theta})}.
\]
Since $\sinh(r+d) \leq \frac{e^{r+d}}{2}$ and $p \geq 3$, it is enough to have
\begin{equation} \label{eq:rhs}
r+d \leq \frac{1-\theta}{2} \log(p) - C(\lambda)
%r+d \leq \frac{1-\theta}{2} \log(p) - \left( \log(2) - \frac{1}{2} \log(1-p^{-2}) \right)
\end{equation}
%Therefore, since $p \geq 3$ and $p \geq 3$, it suffices to have
%\begin{equation} \label{eq:rhs}
%r+d \leq \frac{1-\theta}{2} \log(p) - C_1
%where $C_1 = \log(2) - \frac{1}{2} \log(1-3^{-2}) \simeq 0.75$.
where $C(\lambda) = \frac{1}{2} \log\left(\frac{2(8 \nu \lambda^3+3^{-2-\theta})}{(1-3^{-2}) \nu}\right)$.
Note that this condition immediately implies that 
\[
r+d \leq \log(p) \leq \frac{1}{4} \sys(S_{g_p})
\]
for $p$ large enough, as previously required. \\

Now, in order to show that the surface~$\bar{S}_{g_p+k}$ satisfies the logarithmic systolic growth of~\eqref{eq:loggrowth}, we will need the following lemma.

\begin{lemma} \label{lem:sysmin}
\[
\sys(\bar{S}_{g_p + k}) \geq \min\{\sys(S_{g_p}),2 \pi \sinh(r),2d \}.
\]
\end{lemma}

\begin{proof}
By the circle length formula~\eqref{eq:length}, the closed geodesics of~$\bar{S}_{g_p + k}$ corresponding to a multiple of~$\partial D(x_i,r)$ are of length at least
\[
\length \partial D(x_i,r) = 2 \pi \sinh(r).
\]
The closed geodesics of~$\bar{S}_{g_p + k}$ intersecting~$\partial D(x_i,r)$, but non homotopic to a multiple of~$\partial D(x_i,r)$, have to leave~$D(x_i,r+d)$, and thus, are of length at least~$2d$.
Finally, the closed geodesics of~$\bar{S}_{g_p+k}$ which do not intersect~$\partial D(x_i,r)$ lie in the original surface~$S_{g_p}$, and thus, are of length at least~$\sys(S_{g_p})$.
This implies the desired systolic lower bound.
\end{proof}

At this point, we can fix the values of~$r$ and~$d$ as follows.
Take the real~$r$ so that $2 \pi \sinh(r) = 4 \log(p)$.
That is,
\[
r = \arcsinh\left( \tfrac{2}{\pi} \log p \right) \simeq \log \log p.
\]
Take also the real~$d$ so that the equality case is attained in the equality~\eqref{eq:rhs}.
That is,
\[
d = \frac{1-\theta}{2} \log(p) - r - C(\lambda) \simeq \frac{1-\theta}{2} \log(p).
\]
By Lemma~\ref{lem:sysmin}, we deduce that 
\[
\sys(\bar{S}_{g_p + k}) \gtrsim (1-\theta) \log(p).
\]
For $k \leq \lambda p^\theta$, we obtain that $g_p + k \simeq \nu p^3 $.
Thus,
\[
\sys(\bar{S}_{g_p + k}) \gtrsim \frac{1-\theta}{3} \log(g_p + k).
\]

Since every genus~$g$ large enough can be decomposed as $g=g_p+k$, where $p$ is an odd prime and $k \leq \lambda p^\theta$, the maximum of the systole over all closed $\textrm{CAT}(-1)$ surfaces~$\bar{S}_g$ of genus~$g$ is roughly at least~$\frac{1-\theta}{3} \log g$.
But the maximum of the systole over all closed $\textrm{CAT}(-1)$ surfaces of genus~$g$ is attained by a hyperbolic metric; see~\cite{JS}.
Hence the result.
\end{proof}

\begin{remark}
Our argument applies to surfaces with sparser genera than the ones of Brooks/Buser--Sarnak (\eg, for polynomial genera in primes in an arithmetic progression).
Let $a$ and~$q$ be two coprime numbers.
Suppose that $\{ S_{g_p} \mid p \mbox{ prime with } p \equiv a \, (\mbox{mod } q) \}$ is a family of closed hyperbolic surfaces of genus~$g_p$ with
\[
\sys(S_{g_p}) \geq C_1 \, \log(g_1) - C_2
\]
for some positive constants~$C_1$ and~$C_2$, where $g_p = P(p)$ is polynomial in~$p$ as in~\eqref{eq:gp}.
We can apply the same arguments as in the proof of Theorem~\ref{theo:main} to construct a family of closed hyperbolic surfaces~$S_g$ for every genus~$g$ with a logarithmic systolic growth using the prime gap theorem version of~\cite[Theorem~3]{BHP97} (see the MathSciNet review of the article for a statement without misprints) instead of Theorem~\ref{theo:primegap}.
\end{remark}

\section*{Acknowledgments}

Mikhail Katz was supported by the BSF grant 2020124 and the ISF grant
743/22.  St\'ephane Sabourau was supported by the ANR project Min-Max
(ANR-19-CE40-0014).

\end{document}